\documentclass[12pt,leqno]{article} 
\usepackage{amsmath, amssymb}
\usepackage{amsthm} 
\usepackage{color}
\usepackage{enumerate}
\usepackage{comment}

\theoremstyle{plain} 
\newtheorem{theorem}{\noindent\bf Theorem}[section] 
\newtheorem{lemma}[theorem]{\noindent\bf Lemma}

\theoremstyle{definition} 

\newtheorem{remark}[theorem]{\noindent\sc Remark}

\newtheorem*{acknowledement}{\noindent\sc Acknowledement}
\makeatletter
\def\address#1#2{\begingroup
\noindent\parbox[t]{7.8cm}{%
\small{\scshape\ignorespaces#1}\par\vskip1ex
\noindent\small{\itshape E-mail}%
\/: #2\par\vskip4ex}\hfill%
\endgroup}%

\makeatletter
\@addtoreset{equation}{section}

\makeatother

\makeatother
\title{\Large \bf \uppercase
{Global existence and blow-up for semilinear damped wave equations in three 
space dimensions}}
\author{}
\date{}

\begin{document}
\maketitle

\vspace{-2cm}
\begin{center}
{\textsc{Masakazu Kato}}\\
Muroran Institute of Technology\\
27-1 Mizumoto-cho, Muroran 050-8585, Japan\\
\vspace{5mm}
{\textsc{Miku Sakuraba}}\\
Sapporo Hokuto High School\\
1-10 Kita 15, Higashi 2, Higashi-ku, Sapporo 065-0015, Japan\\
\end{center}

\footnote{
AMS Subject Classifications:\ 35L71, 35E15, 35A01.
}
\begin{abstract}
We consider initial value problem for semilinear damped wave equations in three
space dimensions. We show the small data global existence for the problem without
the spherically symmetric assumption and obtain the sharp lifespan of the solutions. 
This paper is devoted to a proof of the Takamura's conjecture in \cite{DLR15}
on the lifespan  of solutions.
\end{abstract}

\medskip

Keywords: Semilinear damped wave equation, Blow-up, Lifespan, three space dimensions.

\medskip

\section{Introduction}
In this paper, we consider the Cauchy problem for semilinear damped wave equations
\begin{align}\label{1.1}
& v_{tt}(x, t) -\Delta v(x,t) + \frac{\mu}{(1+t)^{\beta}} v_{t}(x,t) = |v(x, t)|^p \quad \mbox{for} \quad (x, t) \in \mathbb{R}^n \times [0, \infty),  \\
\label{1.2}
&v(x, 0) =\epsilon  f(x), \quad v_t(x, 0) =\epsilon g(x) \quad \mbox{for} \quad  x \in \mathbb{R}^n,
\end{align}
where $n \in \mathbb{N}$, $p > 1$, $\beta \in \mathbb{R}$, $\mu>0$ and $\epsilon> 0$. Let $\rho \geq 1$ and we assume that
\begin{align}
	\mbox{supp} \{ f(x), g(x) \} \subset 
	\{ x \in \mathbb{R}^n \ | \ |x|\leq \rho \}.\label{dai1-3}
\end{align}
\par
When $\beta=0$, (\ref{1.1}) is a important mathematical model to describe the wave propagation with
friction and the heat conduction with finite speed of propagation. For example,
magnetohydrodynamics, viscoelasticity, and flood flow with friction are expressed by the damped wave equations. When $\beta \neq 0$, the equation is a model to 
describe the microwave drying processes for hygroscopic materials and 
the vibrations of a strings with variable tension and density, such as the
time-dependent telegraph equation (see \cite{AJM97}). 
\par
It is interesting to find a critical exponent $p_{c}(n)$ such that if $p> p_{c}(n)$, 
then the small data global existence  holds for (\ref{1.1}) and (\ref{1.2}), while if
$1<p \leq p_{c}(n)$, then small data blowup occurs. 
Moreover, for the blowup case, our purpose is to derive  estimates of upper and lower 
bounds of  the  lifespan which is the maximal existence time of the solution.
In this paper, we study small data global existence and blowup for (\ref{1.1}) and (\ref{1.2}) especially with $n=3$, $\beta=1$ and $\mu=2$.
Before we proceed to our problem, we recall some known results.
\par
For the case $\beta \in (-1,1)$, the global existence has been obtained by
D'Abbicco, Lucente and Ressig \cite{DLR13} showed the global existence for
\begin{align*}
	p_{F}(n) < p <
	\left\{
	\begin{array}{ll}
	\infty & \mbox{for} \ n=1,2,\\
	n/(n-2) & \mbox{for} \ n\geq3,
	\end{array}
	\right.
\end{align*}
where $p_{F}(n)=1+2/n$ is the critical power for semilinear heat equation,
$u_{t}-\Delta u=|u|^p$.
For $\beta \in (-1,1)$ and $1 <  p \leq p_{F}(n)$, the lifespan estimates were obtained by Ikeda and Ogawa \cite{IkeOga}, Fuhiwara, Ikeda and Wakasugi \cite{FIW}, Ikeda and Inui \cite{IkeInui}. Then, it is known that the critical exponent is $p_{F}(n)$ for 
$\beta \in (-1,1)$.
\par
When $\beta=1$, the coefficient $\mu$ plays a crucial role in this case.
Indeed, the critical exponent depends on $\mu$. D'Abbicco and Lucente \cite{DL13}, and D'Abbicco \cite{D15} have showed  that the critical power is 
$p_{F}(n)$ when
\begin{align*}
	\mu \geq 
	\left\{
	\begin{array}{ll}
	5/3 & \mbox{for} \ n=1,\\
	3 & \mbox{for} \ n=2,\\
	n+2 & \mbox{for} \ n \geq 3,
	\end{array}
	\right.
\end{align*}
while Wakasugi \cite{Wakasugi14} has obtained blow up for $1< p \leq p_{F}(n)$ and $\mu \geq 1$,
or $1<p\leq p_{F}(n+\mu-1)$ and $0<\mu<1$. We see that if $0< \mu <1$, then $p_{c}(n) \geq p_{F}(n+\mu-1) > p_{F}(n)$.
\par
When $\beta=1$ and $\mu =2$, by setting $u(x,t)=(1+t)v(x,t)$, we can rewrite (\ref{1.1}) and
(\ref{1.2}) as the following semilinear wave equations
\begin{align}
& \Box u(x, t) = (1+t)^{-(p-1)}|u(x, t)|^p \quad \mbox{for} \quad (x, t) \in \mathbb{R}^n \times [0, \infty),\label{dai1-6}  \\
&u(x, 0) =\epsilon  f(x), \quad u_t(x, 0) =\epsilon \{f(x) + g(x)\} \quad \mbox{for} \quad  x \in \mathbb{R}^n.\label{dai1-7}
\end{align}
Due to this observation, D'Abbicco, Lucente and Ressig \cite{DLR15} have determined a critical power
\begin{align}
	p_c(n)= \max \{ p_{F}(n),\ p_{S}(n+2) \} \quad \mbox{for} \ n \leq 3,\label{dai1-10}
\end{align}
where $p_{S}(n)$ is called Strauss exponent which is the critical exponent of semilinear wave equations $w_{tt}- \Delta w = |w|^p$. 
We remark that 
\begin{align*}
	p_{S}(n)=\frac{n+1+\sqrt{n^2+10n-7}}{2(n-1)} \quad (n \geq 2),
\end{align*}
and $p_{S}(n)$ be the positive root of the quadratic equation
\begin{align}
	\gamma(p,n):= 2+(n+1)p-(n-1)p^2=0.\label{dai1-5}
\end{align}
D'Abbicco and Lucente \cite{DL15} have also showed the global existence for 
$p_{S}(n+2) < p < 1+2/\max \{2, (n-3)/2 \}$ to odd and higher dimensions ($n \geq 5$)
under the spherically symmetric assumption. For $n=3$,
$p_{c}(3)=p_{S}(5)$ follows from (\ref{dai1-10}).
In the blowup case $1< p \leq p_{F}(n)$, Wakasugi \cite{Wakasugi14, Wakasugi14-1} has showed that the upper bound of life span of the solutions for (\ref{dai1-6}) and
 (\ref{dai1-7}) is
\begin{align*}
	T(\epsilon) \leq C \epsilon^{-(p-1)/\{2-n(p-1)\}}.
\end{align*}
 In \cite{DLR15}, for $n=3$ and $(f,g) \neq (0,0)$, they remark the following Takamura's conjecture.
\begin{align}
	T(\epsilon) \sim \left\{
\begin{array}{ll}
C\epsilon^{-\frac{2p(p-1)}{\gamma(p,5)}} & (1<p<p_{S}(5))\\
\exp{(C \epsilon^{-p(p-1)})} & (p=p_{S}(5))
\end{array}
\right.\label{dai1-11}
\end{align} 
Our main goal in this paper is to obtain the lifespan estimate (\ref{dai1-11})
for (\ref{dai1-6}) and (\ref{dai1-7}) with $n=3$. Also, our purpose is to show the small data
global existence for $p>p_{S}(5)$ without the symmetric condition.
\par
We put
\begin{align}
	m(p)=\left\{
	\begin{array}{ll}
	1 & (1<p<2)\\
	2 & (p \geq 2)
	\end{array}
	\right..\label{dai1-8}
\end{align}
We think of  $C^{m(p)}$-solutions of the following integral equation
associated with (\ref{dai1-6}) and (\ref{dai1-7}):
\begin{align}
	u(x,t)= u^{0}(x,t) + L[|u|^p](x,t), \quad (x,t) \in \mathbb{R}^3 \times [0,\infty),
	\label{dai2-33}
\end{align}
where
\begin{align}\label{17c}
	L[w](x, t) = \frac{1}{4\pi} \int_{0}^{t} (t-s) ds \int_{|\eta|=1} 
	(1+s)^{-(p-1)}w(x+(t-s)\eta, s) d\omega_{\eta}
\end{align}
for $w \in C(\mathbb{R}^3 \times [0,\infty))$ and 
$u^0$ is a solution to the linear wave equations 
\begin{align}\label{eq:01}
	&u_{tt}(x, t) - \Delta u(x,t) = 0, \quad (x, t) \in \mathbb{R}^3 \times [0, \infty), \\
	\label{eq:02}
	&u(x, 0) = \epsilon f(x), \quad u_t(x, 0) = \epsilon \{ f(x)+g(x) \}, \quad x \in \mathbb{R}^3.
\end{align}
We remark that 
if $u\in C^2(\mathbb{R}^3\times [0,\infty))$ is the solution of (\ref{dai2-33})
with $p \geq 2$, then 
$u$ is the classical solution to the initial value problem (\ref{dai1-6}) and (\ref{dai1-7})
(See Lemma I in \cite{John79}.).

To state our results, we define the lifespan $T(\epsilon)$ of the solution of (\ref{dai2-33}) by
\begin{align*}
	T(\epsilon):=\sup \{ T \in [0,\infty)  \ | \ \mbox{There exists a unique solution}
	\ u \in C^{1}( \mathbb{R}^3 \times [0,T)) \ \mbox{of} \ (\ref{dai2-33}). \}
\end{align*} 
for arbitrarily fixed $(f,g)$.

In the following theorem, we establish the global existence without the spherically symmetric
assumption.
\begin{theorem}\label{thm1}
Let $n=3$, $p > p_{S}(5)=\frac{3+\sqrt{17}}{4}$, $f \in C_{0}^{2+m(p)}(\mathbb{R}^3)$ and $g \in C_{0}^{1+m(p)}(\mathbb{R}^3)$,
where $m(p)$ is given by (\ref{dai1-8}).
If $\epsilon$ is sufficiently small, then
(\ref{dai2-33}) has a unique global solution $u \in C^{m(p)}(\mathbb{R}^3 \times \bigl[ 0,\infty \bigr))$. 
\end{theorem}
\begin{remark}
	In \cite{DLR15}, for $p>p_{S}(5)$, they have showed the global existence in
	$C(\mathbb{R}^3 \times [0,\infty))  \cap C^2(\mathbb{R}^3 
\backslash \{ 0 \} \times [0,\infty))$
	with the radial symmetric condition.
\end{remark}
We obtain the lower bound of the lifespan as follows.
\begin{theorem}\label{thm2}
Let $n=3$, $1< p \leq p_{S}(5)$, $f \in C_{0}^{3}(\mathbb{R}^3)$ and $g \in C_{0}^{2}(\mathbb{R}^3)$. There exist positive constants  $A$ and  $\epsilon_0$ such that
the solution $u \in C^{1}(\mathbb{R}^3 \times [0,\infty))$ of (\ref{dai2-33}) exists as
far as
\begin{align}
T \leq \left\{
\begin{array}{ll}
A\epsilon^{-\frac{2p(p-1)}{\gamma(p,5)}} & (1<p<p_{S}(5))\\
\exp{(A\epsilon^{-p(p-1)})} & (p=p_{S}(5))
\end{array}
\right.
.\label{dai1-1}
\end{align}
for $0 < \epsilon \leq \epsilon_0$.
\end{theorem}
The following theorem shows the optimality of the lower bound of Theorem \ref{thm2}.
Hence, we solve the Takamura's conjecture. 
\begin{theorem}\label{thm3.2}
Let $n=3$, $1 < p \leq p_S(5)$, $f \in C_{0}^{3}(\mathbb{R}^3)$, $g \in C_{0}^{2}(\mathbb{R}^3)$.
We assume that $f \equiv 0$ and $g\geq 0$ $(g \not\equiv 0)$.
There exists a positive constant $\epsilon_0$ such that for $0 < \epsilon \leq \epsilon_0$, the solution of (\ref{dai2-33}) blows up in a finite time $T(\epsilon)$.
Moreover, there exists a positive constant $B$ independent of $\epsilon$
such that
\begin{align}
T(\epsilon) \leq \left\{
\begin{array}{ll}
B\epsilon^{-\frac{2p(p-1)}{\gamma(p,5)}} & (1<p<p_{S}(5))\\
\exp{(B\epsilon^{-p(p-1)})} & (p=p_{S}(5))
\end{array}
\right.
.\label{dai1-2}
\end{align}
\end{theorem}
\begin{remark}
Recently, Ikeda and Sobajima \cite{IkeSoba} obtain improvements on the lifespan estimates for (\ref{dai1-6}) and (\ref{dai1-7}) with $n \geq 1$. In particular,
 when $n=3$, by using the test function method, they have showed the estimates as  follows.
 \begin{align*}
T(\epsilon) \leq \left\{
\begin{array}{ll}
	C\epsilon^{-1-\delta} & (p_{F}(3) < p <p_{S}(7))\\
	C\epsilon^{-\frac{2p(p-1)}{\gamma(p,5)}-\delta} & (p_{S}(7) \leq p<p_{S}(5))\\
	\exp{(C\epsilon^{-p(p-1)})} & (p=p_{S}(5))
\end{array}
\right.
\end{align*}
with  arbitrary small $\delta >0$.
\end{remark}
\begin{remark}
Wakasa \cite{Wa16} had obtained the optimal life span estimate for (\ref{dai2-33}) with $n=1$.
	The study of the lifespan of solutions to the semilinear damped wave equation
with the general variable coefficient has been studied by many mathematicians
(see \cite{LaiTaka18, LTW17} and refer to the references cited therein). 
\end{remark}
The article is organized as follows. 
In Section 2, we show the decay estimate for the solution of linear wave equations
and give the estimate of the integral operator (\ref{17c}). Then, we prove Theorem \ref{thm1}
and Theorem \ref{thm2}. The global existence and lower bounds of the life span will be obtained by the weighted $L^{\infty}$-$L^{\infty}$ estimates introduced in Lemma \ref{lemma2} and Lemma \ref{lem4-2}. In Section 3, we prove Theorem \ref{thm3.2}.
We will show the upper bounds of lifespan by using the iteration argument and the slicing method.
\section{Global existence and Lower bound of the lifespan}

In this section, we prove Theorem \ref{thm1} and Theorem \ref{thm2}. We shall construct a solution of integral equation (\ref{dai2-33}), employing the iteration
method in \cite{John79} and \cite{IKTW}.
\par
First, we study the decay estimate of solution for the homogeneous wave equation (\ref{eq:01}) and (\ref{eq:02}). The solution can be expressed by
\begin{align*}
	u^0(x, t) = \epsilon \frac{t}{4\pi} \int_{|\xi| = 1}  \{ f(x+t\xi)+g(x+t\xi) \} d\omega_{\xi} 
	+ \epsilon \frac{\partial}{\partial t}
	\left\{
 	\frac{t}{4\pi} \int_{|\xi| = 1} f(x+t\xi) d\omega_{\xi}
 	\right\}.
\end{align*}
We prepare the following decay estimate of solution of free wave equations.
\begin{lemma}\label{lem2.3}
Assume a support property (\ref{dai1-3}). Then, there exists a constant $\gamma$ such that for $\nu = 0, 1, 2, 3$ we have
\begin{align}\label{45c}
	\rho^{|\alpha|} |D^{\alpha}_{x} u^0(x, t) | \leq \epsilon \gamma 
	\left( \frac{t+\rho}{\rho} \right)^{-1} N_{\nu} 
	\quad \mbox{for} \quad x \in \mathbb{R}^3, \; t \geq 0, \; |\alpha| \leq \nu,
\end{align}
where
\begin{align*}
	N_{\nu} 	
	&=  \sum_{\substack{ |\alpha| \leq \nu+2}}
	\sup_{x \in \mathbb{R}^3} \rho^{|\alpha|}|D_x^{\alpha}f(x)|+
	\sum_{\substack{|\beta| \leq \nu+1}}
	\sup_{x \in \mathbb{R}^3} 
	\rho^{|\beta|+1} \{|D_x^{\beta}f(x)|	+ |D_x^{\beta}g(x)| \}.
\end{align*}
Moreover, it holds 
\begin{align}\label{45d}
	u^0(x, t) = 0 \quad \mbox{for} \ \ |t-|x|| \geq \rho.
\end{align}
\end{lemma}
\begin{proof}
This well-known fact can be found in (45c) and (45d) of \cite{John79}. We shall omit the proof.
\end{proof}
\par
Next, we estimate the integral operator (\ref{17c}).
We denote a weighted $L^{\infty}$ norm by
\begin{align}\label{4.15}
	\|u\| &= \sup_{\substack{x \in \mathbb{R}^3 \\ 0 \leq t < T}}
	w(|x|,t) |u(x, t)|.
\end{align}
Here,
\begin{align}
	w(r,t)=
	\left\{
	\begin{array}{ll}
	\rho^{-2(p-1)}(t+r+2\rho)^{q}(t-r+2\rho)^{\overline{q}} & (p \neq\frac{3}{2})\\
	\rho^{-1}(t+r+2\rho)\left\{ \log \frac{2(t+r+2\rho)}{t-r+2\rho} \right\}^{-1}
	& (p=\frac{3}{2})\\
	\end{array}
	\right.,\label{4.15-1}
\end{align}
where
\begin{align}
	q=\max \{ 2(p-1),1 \} \quad \mbox{and} \quad \overline{q}=\max \{ 0,2p-3 \}.
	\label{dai4-23}
\end{align}
We define $\displaystyle \overline{v}(r, t) = \sup_{\substack{x\\ |x|=r}} |v(x, t)|$. From (\ref{4.15}) and (\ref{4.15-1}), it follows that 
\begin{align}
	\overline{u}(\lambda,s) \leq \| u \| \times
	\left\{ 
	\begin{array}{ll}
	 	\left( \frac{s+\lambda+2 \rho}{\rho}\right)^{-q}
	 	\left( \frac{s-\lambda+2 \rho}{\rho}\right)^{-\overline{q}}
		\quad &(p \neq\frac{3}{2} ) \\
		\left( \frac{s+\lambda+2 \rho}{\rho}\right)^{-1}
		\log \frac{2(s+\lambda+2\rho)}{s-\lambda+2\rho}
		\quad &(p =  \frac{3}{2})
	\end{array} \right..
	\label{dai4.15-2}
\end{align}
The following a priori estimate plays a key role in the iteration method.
\begin{lemma}\label{lemma2}
Let $p>1$. 
Assume that $u \in C(\mathbb{R}^3 \times [0,T))$ with  {\rm supp}  $ u \subset
\{ (x,t) \in \mathbb{R}^3 \times [0,T) \ | \ |x| \leq t+ \rho \}$.
Then, it holds
\begin{align}\label{eq:18}
	\|L[|u|^p]\| \leq C_{1} \rho^2 \| u \|^p D(T),
\end{align} 
where $D(T)$ is defined by
 \begin{align}\label{eq:18-1}
	D(T) = 
	\left\{ 
	\begin{array}{ll}
	 \left(\frac{2T+3\rho}{\rho}\right)^{\gamma(p,5)/2} 
	\quad &(1 < p < p_{S}(5)) \\
	\log \left( \frac{T+2\rho}{\rho} \right)
	\quad &(p =  p_{S}(5)) \\
	1 & (p>p_{S}(5))	
	\end{array} \right..
\end{align} 
\end{lemma}
\begin{proof}
From the assumption in Lemma \ref{lemma2}, we obtain
\begin{align}
	\overline{u}(\lambda,s)=0 \quad \mbox{for} \ \lambda \geq s+\rho.\label{dai4-24}
\end{align}
Since $\rho \geq 1$, we have from Lemma II in \cite{John79} and (\ref{dai4-24})
\begin{align}
	|L[|u|^p](x,t)| &\leq \int_{0}^{t} ds \int_{I} \frac{\lambda}{2r}
	\left( \frac{s+\rho}{\rho} \right)^{-(p-1)} \overline{u}(\lambda,s)^p d \lambda\nonumber\\
	& = \int \int_{\tilde{R}(r,t)} \frac{\lambda}{2r}
	\left( \frac{s+\rho}{\rho} \right)^{-(p-1)} \overline{u}(\lambda,s)^p d \lambda ds,
	\label{dai2.30}
\end{align}
where $r=|x|$, $I=[ |r-t+s|, r+t-s] \cap [|r-t+s|, s+\rho]$ and
\begin{align*}
	\tilde{R}(r,t)=\{ (\lambda,s) \ | \ 0\leq s \leq t,\ \lambda \in I \}.
\end{align*}
It is clear that (\ref{eq:18}) follows from the basic estimate:
\begin{align}
	|L[|u|^p](x,t)| \leq C \rho^2 \| u \|^p w(r,t)^{-1} D(T).\label{dait5}
\end{align}
We show (\ref{dait5}) in the following two sets separately. Put
\begin{align}
	&S_{1}=\{ (r,t) \ | \ t-\rho \leq r \leq t+\rho \},\quad
	S_{2}=\{ (r,t) \ | \ 0 \leq r \leq t-\rho \}.\label{dait6}
\end{align}
\par
First, we consider the case $(r,t) \in S_1$. 
We see
\begin{align}
	1\leq \frac{t-r+2\rho}{\rho}\leq 3
	\quad \mbox{and} \quad
	\frac{t+r+2\rho}{3}\leq t+ \rho
	\quad \mbox{for} \ (r,t) \in S_1.\label{dai4-4}
\end{align}
From (\ref{dai4.15-2}) and (\ref{dai4-4}), we get
\begin{align}
	\overline{u}(\lambda,s) \leq \| u \| \left(\frac{s+\rho}{\rho}\right)^{-q} \eta(t)
	\quad \mbox{for} \ (\lambda,s) \in S_{1},
	\label{dai4-3}
\end{align}
where
\begin{align}
	\eta(t)=
	\left\{
	\begin{array}{ll}
	1 & (p \neq \frac{3}{2})\\
	\log\frac{6(t+\rho)}{\rho} & (p=\frac{3}{2})\\
	\end{array}
	\right. .\label{dait3}
\end{align}
From p.255 in \cite{John79},
for $0 \leq s \leq t$ and $(r,t) \in S_1$, we find
\begin{align}
	\frac{1}{2r}\int_{I} \lambda d \lambda \leq
	\frac{12\rho(s+\rho)}{t+\rho}.\label{dait4}
\end{align}
Noticing that $\tilde{R}(r,t) \subset S_1$ and substituting (\ref{dai4-3})
into (\ref{dai2.30}), we obtain from (\ref{dait4}), (\ref{dai4-23}) and (\ref{dai4-4})
\begin{align}
	L[|u|^p](x, t)
	&\leq \| u \|^p \eta(t)^p \int_0^t ds
	\left( \frac{s+\rho}{\rho} \right)^{-(q+1)p+1}
	\int_{I}\frac{\lambda}{2r} d \lambda \nonumber\\ 
	&\leq \frac{12\rho^2 \|u \|^p \eta(t)^p}{t+\rho} \int_{0}^{t}
	\left( \frac{s+\rho}{\rho} \right)^{-(q+1)p+2}ds\nonumber\\
	&\leq C \rho^2 \| u \|^p \eta(t)^{p+1}
	 \left( \frac{t+\rho}{\rho} \right)^{-1 + \max\{ -2p^2+p+3,0 \}}\nonumber\\
	&\leq C  \rho^2 \| u \|^p  \left( \frac{t+r+2\rho}{\rho}  \right)^{-q}
	\times \left\{
	\begin{array}{ll}
		\left(\frac{T+\rho}{\rho}  \right)^{-2p^2+3p} & (1<p<\frac{3}{2})\\
		\left\{ \log \frac{6(T+\rho)}{\rho} \right\}^{\frac{5}{2}} & (p= \frac{3}{2})\\
		1 & (p>\frac{3}{2})
	\end{array}
	\right..\nonumber
\end{align}
Hence, we get from (\ref{dai4-4}) and (\ref{eq:18-1})
\begin{align}
	L[|u|^p](x, t) &\leq C  \rho^{2}  \|u \|^p 
	\left( \frac{t+r+2\rho}{\rho} \right)^{-q }
	\left( \frac{t-r+2\rho}{\rho} \right)^{- \overline{q}}D(T) \quad \mbox{in} \ S_1.\label{4.9}
\end{align}
Therefore we get (\ref{dait5}) for $(r,t) \in S_{1}$.
\par
Next, we consider the case $(r,t) \in S_2$.
Introducing new variables of integration
\begin{align}
	\alpha = s+\lambda, \quad \beta = s-\lambda.\label{dai2-10}
\end{align}
For $\alpha+\beta \geq 0$ and $\beta \geq -\rho$, we have
\begin{align}
	\frac{s+ \rho}{\rho}
	&\geq \frac{\alpha+\beta+4\rho}{4 \rho}\nonumber\\
	&\geq \frac{\alpha+2\rho}{4 \rho}.\label{dai2-12}
\end{align}
Since $\displaystyle \left| \frac{\partial (\lambda, s)}{\partial(\alpha, \beta)} \right|=1/2$, it follows from (\ref{dai2.30}) and (\ref{dai2-12}) that
\begin{align}
	|L[|u|^p](x, t)|
	&\leq \int_{-\rho}^{t-r} d\beta \int_{t-r}^{t+r} \frac{\alpha - \beta}{8r} 
	\Bigl(\frac{\alpha + 2\rho}{4\rho}\Bigr)^{-(p-1)} 
	\left\{
	\overline{u} \left( \frac{\alpha-\beta}{2}, \frac{\alpha+ \beta}{2} \right) 
	\right\}^p d\alpha
	\label{dai2.31}
\end{align} 
for $0 \leq r \leq t-\rho$.

We divide the proof into two cases,
$p \neq \frac{3}{2}$ and $p=\frac{3}{2}$.

\noindent (i) \ Estimation in the case of $p \neq \frac{3}{2}$.

Substituting (\ref{dai4.15-2}) into (\ref{dai2.31}), we get
\begin{align}\label{eq:16}
	L[|u|^p](x, t)
	&\leq \| u \|^p \int_{-\rho}^{t-r}d \beta \int_{t-r}^{t+r}
	\frac{\alpha-\beta}{8r} \left( \frac{\alpha+2\rho}{4 \rho}\right)^{-(p-1)}
	\left( \frac{\alpha + 2\rho}{\rho} \right)^{-pq}
	\left( \frac{\beta + 2\rho}{\rho} \right)^{-p \overline{q}} d \alpha\nonumber\\
	&\leq 2^{2(p-1)} \rho^{2p^2-p-1}\|u \|^p
	\times \frac{1}{r} 
	\int_{t-r}^{t+r}(\alpha+2\rho)^{-p(q+1)+2} d\alpha \times
	\int_{-\rho}^{t-r} (\beta+ 2 \rho)^{-p \overline{q}} d\beta.
\end{align} 
We evaluate the $\alpha$-integral in (\ref{eq:16}). For the case $t-r+2\rho \geq  \frac{1}{2}(t+r+2\rho)$, we get
\begin{align}
	\frac{1}{r} \int_{t-r}^{t+r} (\alpha+2\rho)^{-p(q+1)+2} d\alpha 
	&\leq \frac{C}{r}\int_{t-r}^{t+r} d\alpha
	\times (t+r+2\rho)^{-p(q+1)+2}
	\nonumber\\
	&\leq 
	\left\{
	\begin{array}{ll}
	C(t+r+2\rho)^{-2p^2+p+2} & (1<p < \frac{3}{2})\\
	C(t+r+2\rho)^{-1}(t-r+2\rho)^{-(2p-3)} & ( p>\frac{3}{2} )
	\end{array}
	\right..\label{2.8}
\end{align}
Next, we consider $t-r+2\rho \leq  \frac{1}{2}(t+r+2\rho)$. In other words, for $t+2\rho \leq 3r$, we obtain
\begin{align}
	\frac{1}{r} \int_{t-r}^{t+r} (\alpha+2\rho)^{-p(q+1)+2} d\alpha
	&\leq \frac{4}{t+r+2\rho} \times
	\left\{ 
	\begin{array}{ll}
		\frac{1}{(-2p^2+p+3)}(t+r+2\rho)^{-2p^2+p+3} & (1<p<\frac{3}{2})\\
		\frac{1}{2p-3}(t-r+2\rho)^{-(2p-3)} & (p>\frac{3}{2})
	\end{array}
	\right..\label{2.9}
\end{align}
It follows from (\ref{2.8}) and (\ref{2.9})
\begin{align}
	\frac{1}{r}\int_{t-r}^{t+r} (\alpha+2 \rho)^{-p(q+1)+2} d \alpha
	\leq \left\{
	\begin{array}{ll}
	C(t+r+2\rho)^{-2p^2+p+2} & (1 < p < \frac{3}{2}) \\
	C (t+r+2\rho)^{-1}(t-r+2\rho)^{-(2p-3)} & (p > \frac{3}{2})
	\end{array}
	\right..\label{4.8}
\end{align}
We evaluate the $\beta$-integral in (\ref{eq:16}). We get from (\ref{dai4-23})
\begin{align}
	\int_{-\rho}^{t-r} (\beta+2\rho)^{-p\overline{q}} d\beta
	\leq \left\{
	\begin{array}{ll}
	t-r+\rho & (1< p < \frac{3}{2})\\
	\frac{2}{\gamma(p,5)}(t-r+2\rho)^{\gamma(p,5)/2} & (\frac{3}{2}<p<p_{S}(5))\\
	\log\frac{t-r+2\rho}{\rho} & (p=p_{S}(5))\\
	C \rho^{\gamma(p,5)/2} & (p>p_{S}(5))
	\end{array}
	\right. .\label{4.25}
\end{align}
Hence, from (\ref{eq:16}), (\ref{4.8}), (\ref{4.25}) and (\ref{eq:18-1}), it follows that
\begin{align}
	L[|u|^p](x,t) \leq C  \rho^{2}  \|u \|^p 
	\left( \frac{t+r+2\rho}{\rho} \right)^{-q }
	\left( \frac{t-r+2\rho}{\rho} \right)^{- \overline{q}}
	D(T)
	\quad \mbox{in} \ S_2.\label{dai4-28}
\end{align}

\noindent (ii) Estimation in the case of $p=\frac{3}{2}$.

Substituting (\ref{dai4.15-2}) into (\ref{dai2.31}), we get
\begin{align}
	L[|u|^p](x, t)
	&\leq \| u \|^{\frac{3}{2}} \int_{-\rho}^{t-r}d \beta \int_{t-r}^{t+r}
	\frac{\alpha-\beta}{8r} \left( \frac{\alpha+2\rho}{4 \rho}\right)^{-\frac{1}{2}}
	\left( \frac{\alpha+2 \rho}{\rho} \right)^{-\frac{3}{2}} 
	\left\{\log \frac{2(\alpha + 2 \rho) }{\beta +2 \rho} \right\}^\frac{3}{2} 
	d \alpha\nonumber\\
	& \leq C \rho^2 \|u \|^{\frac{3}{2}} \times \frac{1}{r} 
	\int_{-\rho}^{t-r}  d\beta \int_{t-r}^{t+r} (\alpha+2\rho)^{-1}
	\left\{ \log \frac{2(\alpha+2\rho)}{\beta+2 \rho} \right\}^{\frac{3}{2}} d\alpha\nonumber\\
	&\leq C \rho^2 \|u \|^{\frac{3}{2}} \left\{ J_{1}+J_{2} \right\},\label{dai4-8}
\end{align}
where
\begin{align*}
	J_{1}&:=
	\left\{ \log \frac{2(t+r+2\rho)}{ \rho } \right\}^{\frac{3}{2}} \times \frac{1}{r} 
	\int_{-\rho}^{\rho}  d\beta \int_{t-r}^{t+r} (\alpha+2\rho)^{-1} d\alpha,\nonumber\\
	J_{2}&:=\frac{1}{r}  
	\int_{\rho}^{t-r}  d\beta \int_{t-r}^{t+r} (\alpha+2\rho)^{-1}
	\left\{ \log \frac{2(\alpha+2\rho)}{\beta+2 \rho} \right\}^{\frac{3}{2}} d\alpha.
\end{align*}
By the same calculation as (\ref{2.8}) and (\ref{2.9}), we have
\begin{align}
	J_{1} &\leq
	C \left\{ \log \frac{2(t+r+2\rho)}{ \rho } \right\}^{\frac{3}{2}}  
	\left( \frac{t+r+2\rho}{\rho} \right)^{-1}
	\log \frac{2(t+r+2\rho)}{ t-r+2\rho }\nonumber\\
	&\leq Cw(r,t)^{-1} \left\{  \log\frac{2(2T+3\rho)}{\rho}\right\}^{\frac{3}{2}}.\label{dai4-9}
\end{align}
Next we evaluate $J_{2}$. Introducing new variables of integration $\sigma$, $\theta$ in $J_2$ by
\begin{align}
	2(\alpha + 2 \rho)= (t-r+2\rho) \sigma, \quad
	\beta+2\rho=(t-r+2\rho) \theta \sigma \nonumber
\end{align}
and setting
\begin{align}
	A=t+r+ 2\rho,\quad B=t-r+2 \rho,\nonumber
\end{align}
we have
\begin{align}
	J_{2}&=\frac{B}{r} \int_{2}^{\frac{2A}{B}} d \sigma \int_{\frac{3\rho}{B \sigma}}^{\frac{1}{\sigma}}
	(-\log \theta)^{\frac{3}{2}} d \theta\nonumber\\
	&\leq \frac{2(A-B)}{r}\int_{0}^{1}
	(-\log \theta)^{\frac{3}{2}} d \theta\nonumber\\
	&\leq C.\label{dai4-10}
\end{align}
We obtain from (\ref{dai4-8}), (\ref{dai4-9}), (\ref{dai4-10}) and (\ref{eq:18-1})
\begin{align}
	L[|u|^p](x,t) \leq C\rho^2 \| u \|^{\frac{3}{2}}w(r,t)^{-1}D(T)
	\quad \mbox{in} \ S_2.\label{dai4-29}
\end{align}
Hence, from (\ref{dai4-28}) and (\ref{dai4-29}), we get (\ref{dait5}) 
for $(r,t) \in S_{2}$.
 This completes the proof.
\end{proof}
\begin{proof}[Proof of theorem \ref{thm1}]
We define
\begin{align*}
	X = \left\{
	u(x, t)  \ \Biggl| \
\begin{array}{ll}
	D_x^{\alpha} u(x, t) \in C(\mathbb{R}^3 \times [0, \infty)), 
	\quad \| D_x^{\alpha} u\| < \infty 	\quad (|\alpha| \leq m(p)) \\
	u(x, t) = 0 \quad (|x| \geq t+\rho) \\
\end{array}
\right\},
\end{align*} 
where $D_x^{\alpha} = D_1^{\alpha_1} D_2^{\alpha_2} D_3^{\alpha_3}$ $(\alpha = (\alpha_1, \alpha_2, \alpha_3))$ and  $D_k = \frac{\partial}{\partial x_k}$ $(k = 1, 2, 3)$.
We can verify easily that $X$ is complete with respect to the norm
\begin{align*}
	\| u \|_{X} =\sum_{| \alpha | \leq m(p)} \| D_{x}^{\alpha}  u\|.
\end{align*}
Using the iteration method, we shall construct a solution of (\ref{dai2-33}). We define the sequence of functions $\{u_n\}$ by
\begin{align*}
	u_0 = u^0, \quad u_{n+1} = u^0 + L[|u_{n}|^p] \quad \mbox{for} \quad n \geq 0.
\end{align*}
We have from (\ref{45c}), (\ref{45d}) and (\ref{4.15})
\begin{align}\label{49b}
	\| D^{\alpha} u^0 \| 
	&\leq  \epsilon \gamma 3^{2(p-1)} \rho^{-\nu} N_{\nu} \quad (|\alpha| = \nu \leq 3).
\end{align}
As in \cite{John79}, p.258, we see from Lemma \ref{lemma2} and (\ref{49b}) that if $u^0$ satisfies
\begin{align}
	C_1\rho^2\| u^0 \|^{p-1} \leq \frac{1}{p2^p},\label{dait8}
\end{align}
then $\{ u_{n} \}$ is a Cauchy sequence in
$X$. Since $X$ is complete, there exists a function $u \in X$ such that $D_{x}^{\alpha} u_n$
converges uniformly to $D_{x}^{\alpha} u$, $n \to \infty$. Clearly $u$ satisfies
 (\ref{dai2-33}). In view of (\ref{dai2-33}) and (\ref{17c}), we note that
$\partial u/ \partial t$ can be expressed in
 $D_{x}^{\alpha} u$ $(|\alpha| \leq 1)$.
Thus, from (\ref{49b}) and (\ref{dait8}), Theorem \ref{thm1} is proved by taking $\epsilon$ is small.
\end{proof}
To prove Theorem \ref{thm2}, the following variant to the a priori estimate
is required.
\begin{lemma}\label{lem4-2}
Let $1 < p  \leq p_{S}(5)$ and $0 \leq \nu \leq p-1$.
Assume that $u^0, u \in C(\mathbb{R}^3 \times [0, T))$ with (\ref{45d})
and {\rm supp}  $ u \subset
\{ (x,t) \in \mathbb{R}^3 \times [0,T) \ | \ |x| \leq t+ \rho \}$.
It holds
\begin{align}
	\| L[|u^0|^{p-\nu}|u|^{\nu}] \| \leq C \rho^2 \| u^0\|_0^{p-\nu} 
	\| u \|^{\nu} \times
	\left\{
	\begin{array}{ll}
	D(T)^{ \frac{2 \nu(-2p+3)}{\gamma(p,5)}} & (1< p < \frac{3}{2})\\
	\left\{ \log \frac{6(T+\rho)}{\rho} \right\}^{\nu} & (p=\frac{3}{2})\\
	1 & (\frac{3}{2}< p \leq p_{S}(5))
	\end{array}
	\right.,\label{dai4-11}
\end{align}
where
\begin{align}
	\| u^0 \|_{0}:=\sup_{\substack{x \in \mathbb{R}^3 \\ 0 \leq t < \infty}}
	\rho^{-1}(t+|x|+2\rho) |u^0(x, t)|.\label{dai4-30}
\end{align}  
\end{lemma}
\begin{proof}
From (\ref{dai4-30}) and (\ref{45d}), we have
\begin{align}
	0\leq \overline{u^0}(\lambda,s) \leq \|  u^0 \|_{0}
	\left( \frac{s+\lambda+2\rho}{\rho} \right)^{-1}\label{dai4-12}
\end{align}
and 
\begin{align}
	\overline{u^0}(\lambda,s)=0 \quad \mbox{for \quad} |\lambda-s|\geq \rho.
	\label{dai4-13}
\end{align}
To obtain (\ref{dai4-11}), we show
\begin{align}
	L[|u^0|^{p-\nu} |u|^{\nu}](x,t)
	\leq C \rho^2 \| u^0 \|_{0}^{p-\nu} \| u \|^{\nu} w(r,t)^{-1} \times
	\left\{
	\begin{array}{ll}
	D(T)^{ \frac{2 \nu(-2p+3)}{\gamma(p,5)}} & (1< p < \frac{3}{2})\\
	\left\{ \log \frac{6(T+\rho)}{\rho} \right\}^{\nu} & (p=\frac{3}{2})\\
	1 & (\frac{3}{2}< p \leq p_{S}(5))
	\end{array}
	\right..\label{dait7}
\end{align}
\par
First, we show (\ref{dait7}) for $(r,t) \in S_1$ defined by (\ref{dait6}). We get
from (\ref{dai4-12}), (\ref{dai4-4}) and (\ref{dai4-3})
\begin{align}
	\overline{u^0}(\lambda,s)^{p-\nu} \overline{u}(\lambda,s)^{\nu}
	\leq \| u^0 \|_0^{p-\nu} \| u \|^{\nu} \eta(t)^{\nu} 
	\left(\frac{s+\rho}{\rho}\right)^{-p+\nu(1-q)} \quad \mbox{in} \ S_{1},
	\label{dai4-25}
\end{align}
where $\eta(t)$ is defined by (\ref{dait3}).\\
By the same way as in (\ref{dai2.30}), we obtain
\begin{align}
	L[|u^0|^{p-\nu}|u|^{\nu}] (x, t) \leq 
	\int \int_{\tilde{R}(r,t)} \frac{\lambda}{2r}
	\left( \frac{s+\rho}{\rho} \right)^{-(p-1)} 
	\overline{u_0}(\lambda,s)^{p-\nu} \overline{u}(\lambda,s)^{\nu} d \lambda ds.
	\label{dait10}
\end{align}
For $(r,t) \in S_1$, noticing $\tilde{R}(r,t) \subset S_1$ and substituting 
(\ref{dai4-25}) into (\ref{dait10}), we obtain from (\ref{dait4})
and (\ref{dai4-4})
\begin{align}
	L[|u^0|^{p-\nu}|u|^{\nu}] (x, t)
	&\leq\|u^0\|_{0}^{p-\nu} \| u \|^{\nu} \eta(t)^{\nu}
	\int_{0}^{t} ds \left(\frac{s+\rho}{\rho}\right)^{-2p+1+\nu(1-q)}
	\int_I \frac{\lambda}{2r}  d\lambda\nonumber\\
	&\leq \frac{12\rho^{2}\|u^0\|_0^{p-\nu} \| u \|^{\nu} \eta(t)^{\nu}}{t+\rho} 
	\int_{0}^{t} \left( \frac{s+\rho}{\rho} \right)^{-2(p-1)+\nu(1-q)} ds  \nonumber  \\
	&\leq C\rho^{2} \|u^0\|_0^{p-\nu} \| u \|^{\nu} \eta(t)^{\nu+1}
	 \left( \frac{t+\rho}{\rho} \right)^{-1
	+\max\{-2p+3+\nu(1-q),0 \}}\nonumber\\ 
	&\leq C \|u^0\|_0^{p-\nu} \| u \|^{\nu}  \rho^2 w(r,t)^{-1}
	\times \left\{
	\begin{array}{ll}
	\left( \frac{T+\rho}{\rho} \right)^{\nu(-2p+3)} & (1 < p < \frac{3}{2})\\
	\left\{ \log \frac{6(T+\rho)}{\rho} \right\}^{\nu} & (p=\frac{3}{2})\\
	1 & (\frac{3}{2} < p \leq p_{S}(5))
	\end{array}
	\right..\label{dai4-14}
\end{align}
Hence we obatin (\ref{dait7}) for $(r,t) \in S_1$.
\par
Next, we show (\ref{dait7}) for $(r,t) \in S_2$ defined by (\ref{dait6}).\\
By the same way as in (\ref{dai2.31}), we get from (\ref{dai4.15-2}), (\ref{dai4-12}) and (\ref{dai4-13})
\begin{align}
	L[|u^0|^{p-\nu} |u|^{\nu}] (x, t) 
	&\leq  \|u^0\|_0^{p-\nu} \| u \|^{\nu} \eta(t)^{\nu}
	\int_{-\rho}^{\rho} d\beta \int_{t-r}^{t+r} 
	\frac{\alpha - \beta}{8r} 
	\Bigl(\frac{\alpha+2\rho}{4\rho} \Bigr)^{-(p-1)}\nonumber\\
	& \times\Bigl( \frac{\alpha + 2\rho}{\rho} \Bigr)^{-(p-\nu)-q \nu}
	\Bigl( \frac{\beta + 2\rho}{\rho} \Bigr)^{- \overline{q}\nu} 
	d\alpha\nonumber\\
	&\leq 2^{2p-5}\rho^{2p-1-\nu(1-q-\overline{q})} \|u^0\|_0^{p-\nu} \|u \|^{\nu}\eta(t)^{\nu}
	\times \frac{1}{r}\int_{t-r}^{t+r} (\alpha +2\rho)^{-2(p-1)+\nu(1-q)} 
	d\alpha\nonumber\\
	& \times \int_{-\rho}^{\rho} (\beta +2 \rho)^{- \overline{q} \nu} d\beta 
	\nonumber\\
	&\leq C \rho^{2p-\nu(1-q)} \|u^0\|_0^{p-\nu} \| u \|^{\nu} \eta(t)^{\nu}
	\times \frac{1}{r}\int_{t-r}^{t+r} (\alpha +2\rho)^{-2(p-1)+\nu(1-q)} 
	d\alpha.\label{dai4-26} 
\end{align}
By the same calculation as  (\ref{2.8}) and (\ref{2.9}), we obtain
\begin{align}
	\frac{1}{r} \int_{t-r}^{t+r} (\alpha+2\rho)^{-2(p-1)+\nu(1-q)} d\alpha
	\leq
	\left\{
	\begin{array}{ll}
		 C(t+r+2\rho)^{-2(p-1)+\nu(1-q)}& 
		 (p \neq \frac{3}{2})\\
		C(t+r+2\rho)^{-1} \log \frac{2(t+r+2\rho)}{t-r+2\rho} & (p=\frac{3}{2})
	\end{array}
	\right..\label{dai4-27}
\end{align}
We have from (\ref{dai4-26}) and (\ref{dai4-27})
\begin{align}
	L[|u^0|^{p-\nu}  |u|^{\nu}] (x, t) 
	&\leq C \rho^2  \|u^0\|_0^{p-\nu} \| u \|^{\nu}   w(r,t)^{-1} 
	\times \left\{
	\begin{array}{ll}
		\left( \frac{2T+3\rho}{\rho} \right)^{\nu(-2p+3)} & (1 < p < \frac{3}{2})\\
		\left\{ \log \frac{6(T+\rho)}{ \rho } \right\}^{\nu} & (p=\frac{3}{2})\\
		1 & (\frac{3}{2} < p \leq p_{S}(5))
	\end{array}
	\right..\label{dai4-15}
\end{align}
Therefore, we get (\ref{dait7}) for $(r,t) \in S_2$.
This completes the proof.
\end{proof}
From Lemma \ref{lem4-2}, we obtain the following estimate.
\begin{lemma}\label{lem4-3}
Suppose that the assumptions in Lemma \ref{lem4-2} are fulfilled.
Then we have
\begin{align}
	\| L[|u^0|^{p-\nu} |u|^{\nu}] \|
	\leq C_{2} \rho^2 \| u^0 \|_{0}^{p-\nu} \| u \|^{\nu}
      \times
      \left\{
      \begin{array}{ll}
      1  & (\nu=0)\\
      D(T)^{1/p} & (\nu=1)\\
	D(T)^{(p-1)/(p+1)} & (\nu=p-1)
      \end{array}
      \right..\label{dai4-21}
\end{align}
\end{lemma}

\begin{proof}
When $\frac{3}{2} \leq p \leq p_{S}(5)$, from Lemma \ref{lem4-2} and (\ref{eq:18-1}),
the estimate (\ref{dai4-21}) is trivial. 
When $1< p< \frac{3}{2}$, we see that $\frac{2(-2p+3)}{ \gamma (p,5) } \leq \frac{1}{p+1}$.
Hence, from Lemma \ref{lem4-2} and (\ref{eq:18-1}), we get (\ref{dai4-21}). This completes the proof.
\end{proof}
\begin{proof}[Proof of Theorem \ref{thm2}]

We consider the following integral equation:
\begin{align}
	U=L[|u^0+U|^p] \quad \mbox{in} \quad \mathbb{R}^3 \times [0,T).
	\label{dai4-18}
\end{align}
Suppose we obtain the solution of (\ref{dai4-18}). Then, putting $u=U+u^0$, we get the solution
of (\ref{dai2-33}), and its lifespan is the same as that of $U$. Thus we have reduced the problem
to the analysis of (\ref{dai4-18}).
Let $Y$ be the norm space defined by
\begin{align*}
	Y = \left\{
	U(x, t) \ \Biggl| \
\begin{array}{ll}
	D_x^{\alpha} U(x, t) \in C(\mathbb{R}^3 \times [0, T)), 
	\quad \| D_x^{\alpha} U\| < \infty 	\quad (|\alpha| \leq 1) \\
	U(x, t) = 0 \quad (|x| \geq t+\rho) \\
\end{array}
\right\},
\end{align*}
which is equipped with the norm
\begin{align*}
	\| U \|_{Y}= \sum_{ |\alpha| \leq 1 } \| D_{x}^{\alpha} U \|. 
\end{align*}
We shall construct a solution of the integral equation (\ref{dai4-18}) in $Y$. We define the sequence of functions $\{ U_{n} \}$ by
\begin{align*}
	U_{0}=0,\quad U_{n+1}=L[|u^0+U_{n}|^p] \quad (n=0, 1, 2, \cdots).
\end{align*}
From Lemma \ref{lem2.3}, (\ref{dai4-4}) and (\ref{dai4-30}), we see that there exists a positive constant $C_0$ such that
\begin{align*}
	\| D_{x}^{\alpha} u^0 \|_{0} \leq C_{0} \epsilon \quad (|\alpha| \leq 1).
\end{align*}
We put
\begin{align*}
	C_{3}:=(2^{2p+2}p)^{\frac{p}{p-1}}
	\max \left\{ C_{1} \rho^2 M_{0}^{p-1}, (C_{2} \rho^2 C_{0}^{p-1})^p,
	(C_{2} \rho^2 M_{0}^{p-2} C_{0})^{ \frac{p}{p-1}} \right\}
\end{align*}
and
\begin{align*}
	M_{0}:=2^p p \rho^2 C_{0}^p \max \{C_{1}, C_{2} \},
\end{align*}
where $C_1$ and $C_2$ are constants given in Lemma \ref{lemma2} and Lemma 
\ref{lem4-3}.
Analogously to the proof of Theorem 1 in \cite{IKTW},
from Lemma \ref{lemma2} and Lemma \ref{lem4-3}, we see that $\{ U_n \}$ is a Cauchy sequence in $Y$, if 
\begin{align}
	C_{3} \epsilon^{p(p-1)} D(T) \leq 1\label{dait9}
\end{align}
holds (see (4.3) in \cite{IKTW}).
We can verify easily that $Y$ is complete. Hence, there exists a function $U$ such that $U_n$ converges to $U$ in $Y$. Therefore $U$ satisfies the integral equation (\ref{dai4-18}).\\
Let us fix $\epsilon_{0}$ as 
\begin{align}
	C_{3} \epsilon_{0}^{p(p-1)}  \leq
	\left\{
	\begin{array}{ll}
		6^{- \gamma(p,5)/2} & (1<p<p_{S}(5))\\
		(\log 4)^{-1} & (p=p_{S}(5))
	\end{array}
	\right..\label{dai4-19}
\end{align}
For $0 < \epsilon \leq \epsilon_{0}$, if we assume that
\begin{align}
	C_{3} \epsilon^{p(p-1)}\leq
	\left\{
	\begin{array}{ll}
	 \left( \frac{4T}{\rho} \right)^{-\gamma(p,5)/2} \ (1<p<p_{S}(5))\\ 
	\left( \log \frac{2T}{\rho} \right)^{-1} \ (p=p_{S}(5))
	\end{array}
	\right.,\label{dai4-20}
\end{align}
then (\ref{dait9}) holds.
Hence the lower bound estimate (\ref{dai1-1}) follows immediately 
from (\ref{dai4-20}). This completes the proof.
\end{proof}
\section{Upper bound of the lifespan}

In this section, we prove Theorem \ref{thm3.2}. Our proof is based on the iteration argument which was introduced by \cite{John79}.
For the critical case, we apply the slicing method which was introduced by \cite{RYH00}. We show that the soluiton of (\ref{dai2-33}) blows up in finite time.
\par
Let $\widetilde{v}(r,t)$ be the spherical mean of 
$v \in C(\mathbb{R}^3 \times [0,\infty))$;
\begin{align}\label{18}
\widetilde{v}(r, t) = \frac{1}{4\pi} \int_{|\xi| = 1} v(r\xi, t) d\omega_{\xi}.
\end{align}
We get the following representation formula (\ref{19}) (for the proof, see \cite{John79}).
\begin{lemma}\label{lem3}
Let $L$ be the linear integral operator defined by (\ref{17c}). Assume that $v \in C(\mathbb{R}^3 \times [0,T))$. Then, For $(r,t) \in [0,\infty) \times [0,T)$, it holds
\begin{align}\label{19}
\widetilde{L[v]}(r, t) = \int \int_{R(r, t)} \frac{\lambda}{2r} (1+s)^{-(p-1)} \widetilde{v}
	(\lambda, s) d\lambda,
\end{align}
where
\begin{align}
	R(r,t) = \{(\lambda, s) 
	  \ | \ t-r \leq s+\lambda \leq t+r, \; s-\lambda \leq t-r, \;  s\geq0 \}.
	\label{dait2}
\end{align}
\end{lemma}
Since $p > 1$, $|u|^p$ is a convex function with $u$. By using the Jensen's inequality
and Lemma \ref{lem3}, it follows from (\ref{dai2-33}) that
\begin{align}\label{27}
	\widetilde{u}(r, t) 
	&\geq \widetilde{u^0}(r,t) + \int \int_{R(r, t)} \frac{\lambda}{2r} (1+s)^{-(p-1)} |\widetilde{u}
	(\lambda, s)|^p d\lambda  .
\end{align}
We define the following domains:
\begin{align}
	\Sigma_j = \{ (r, t) \; | \; l_j\rho \leq t-r \leq r \},\quad
	\Sigma_\infty = \{ (r, t) \; | \; 2\rho \leq t-r \leq r \},
	\label{dai5-20}
\end{align}
where
\begin{align*}
	l_j = 1 + \frac{1}{2} + \dots + \frac{1}{2^j} \quad (j=0,1,2,\cdots).\nonumber
\end{align*}

We derive a lower bound of the solution to (\ref{dai2-33}), which is a first step of our iteration argument.
\begin{lemma}\label{lem5.7}
We assume that $f \equiv 0$ and $g\geq 0$ $(g \not\equiv 0)$. Let $u$ be a solution of (\ref{dai2-33}). Then there exists a positive constant M independent of $\epsilon$
such that 
\begin{align}
	\tilde{u}(r, t) 
	&\geq \frac{M \epsilon^p}{(t+r)(t-r)^{2p-3}}
	\quad \mbox{in} \quad \Sigma_0.\label{6.2}
\end{align}
\end{lemma}
\begin{proof}
From (\ref{eq:01}) and (\ref{18}),  the spherical averages $\widetilde{u^0}(r, t)$ satisfy $(\frac{\partial^2}{\partial t^2} - \frac{\partial^2}{\partial r^2})r\widetilde{u^0}(r, t) = 0$.
By using  the d'Alembert's formula and the assumption $f \equiv 0$, we obtain
\begin{align}\label{eq:09}
	\widetilde{u^0}(r, t) = \frac{H(t+r) - H(t-r)}{2r},
\end{align}
where
\begin{align*}
	H(s) =  - \epsilon \int_{s}^{\infty} \sigma
	\widetilde{g}(\sigma) d\sigma.
\end{align*}
From (\ref{dai1-3}), we have $H(s) = 0$ for $s \geq \rho$. 
Since $g(x) \geq 0$ $(\not\equiv 0)$, there exist constants $M_0>0$, $a, b$ $(0 < a < b < \rho)$ such that
\begin{align*}
	H(s) \leq -2M_0 \epsilon \quad (a \leq s \leq b).
\end{align*}
From (\ref{eq:09}), for $t+r \geq \rho$, $a \leq t-r \leq b$, we get
\begin{align}
	\widetilde{u^0}(r, t) &= -\frac{H(t-r)}{2r}\nonumber\\
	&\geq \frac{M_{0} \epsilon}{r}.\label{dai5-16}
\end{align}

We put
\begin{align*}
	S(r, t) &= \{(\lambda, s)  \; | \; t-r \leq \lambda, \; s + \lambda \leq 3(t-r), \;  a 
	\leq s-\lambda \leq b  \}.
\end{align*}
For $(r,t) \in \Sigma_0$, we see $t+r \geq 3(t-r)$. Then, it follows that
\begin{align}
	&S(r,t) \subset
	\left\{(\lambda,s) \ | \ s+\lambda \geq \rho, a \leq s-\lambda \leq b \right\}
	\quad \mbox{in} \ \Sigma_0.\label{dai5-22}
\end{align}
Then, for $(r,t) \in \Sigma_0$, we have from (\ref{27}), (\ref{dai5-16})
and (\ref{dai5-22})
\begin{align}
	\widetilde{u}(\lambda, s)
	&\geq \frac{M_{0} \epsilon}{\lambda} \quad \mbox{in} \quad S(r,t).\label{dai5-17}
\end{align}
Noticing that $S(r,t) \subset R(r,t)$ for $(r,t) \in \Sigma_0$ and 
substituting (\ref{dai5-17}) into (\ref{27}), we get
\begin{align*}
	\tilde{u}(r, t) \geq \iint_{S(r, t)} \frac{\lambda}{2r} 
	(1+s)^{-(p-1)} \left|\frac{M_{0} \epsilon}{\lambda}\right|^p d\lambda ds
	\quad \mbox{in} \ \Sigma_0.
\end{align*}
Using the variables of integration $\alpha$, $\beta$ from (\ref{dai2-10})
and since $t-r \geq \rho \geq 1$, we get
\begin{align}
	\tilde{u}(r, t) 
	&\geq \frac{M_{0}^p \epsilon^p}{4r}\int_{a}^{b} d\beta 
	\int_{2(t-r) + \beta}^{3(t-r)} \left(\frac{\alpha - \beta}{2}\right)^{1-p}
	 \left( \frac{2+\alpha + \beta}{2} \right)^{1-p} d\alpha\nonumber\\
	&\geq  \frac{M_0^p \epsilon^p}{4r}\int_{a}^{b} d\beta
	 \int_{2(t-r) + \beta}^{3(t-r)}
	(3(t-r))^{1-p} (6(t-r))^{1-p} d\alpha \nonumber\\
	&=  \frac{18^{-p} M_0^p \epsilon^p }{4r (t-r)^{2(p-1)}}
	\int_{a}^{b} (t-r-\beta) d\beta\nonumber\\
	&\geq  \frac{18^{-p} (b-a) M_0^p \epsilon^p (t-r-b) }{2(t+r) (t-r)^{2(p-1)}}\nonumber\\
	&\geq \frac{M \epsilon^p}{(t+r)(t-r)^{2p-3}}
	\quad \mbox{in} \ \Sigma_0,\nonumber
\end{align}
where $\displaystyle M =  2^{-1}18^{-p} (b-a)\left(1-b/\rho \right)M_{0}^p $.
This completes the proof.
\end{proof}
\begin{proof}[Proof of Theorem \ref{thm3.2}]
For $(r,t) \in \Sigma_0$, we introduce the sets
\begin{align*}
	Q_j(r, t) &= \{(\lambda, s)  \; | \; 
	t-r \leq \lambda, \; s + \lambda \leq 3(t-r), \; 
	l_j \rho \leq s-\lambda \leq t-r \} \quad (j=0,1,2,\cdots).
\end{align*}
For $(r,t) \in \Sigma_0$, we see $t+r \geq 3(t-r)$. Then, we have
\begin{align}
	&Q_j(r,t) \subset R(r,t) \quad \mbox{in} \ \Sigma_0,\label{dai5-24}\\
	&Q_j(r,t) \subset \Sigma_j \quad \mbox{in} \ \Sigma_0,\label{dai5-18}
\end{align}
where $R(r,t)$ and $\Sigma_j$ are defined by (\ref{dait2}) and (\ref{dai5-20}).\\
Since $\Sigma_{j} \subset \Sigma_0$, we have from (\ref{27}) and (\ref{dai5-24})
\begin{align}
	\widetilde{u}(r, t) 
	&\geq \int \int_{Q_j(r,t)} \frac{\lambda}{2r} (1+s)^{-(p-1)} |\widetilde{u}
	(\lambda, s)|^p d\lambda \quad \mbox{in} \ \Sigma_j.\label{6.1}
\end{align}
\par
We divide the proof into two cases, $1<p<p_{S}(5)$ and $p=p_{S}(5)$.\\
\noindent (i) \ Estimation in the case of $1<p<p_{S}(5)$.
\par
We define the sequences $\{ a_j \}, \{ b_j \}$ and $\{ D_j \}$ by
\begin{align}
	a_0=1,\quad b_0=2(p-1),\quad D_{0}=M \epsilon^p\label{40b}
\end{align}
and
\begin{align}
	&a_{j+1}=p a_{j}+2,\quad b_{j+1}=pb_{j}+2(p-1),\quad
	D_{j+1}=\frac{18^{-p}}{2a_{j+1}^2}D_{j}^p
	\quad (j=0,1,2,\cdots).\label{40c}
\end{align}

By using the induction argument, we derive
\begin{align}
	\tilde{u}(r,t) \geq \frac{D_{j}(t-r-\rho)^{a_j}}{(t+r)(t-r)^{b_j}}
	\quad \mbox{in} \ \Sigma_0
	\quad (j=0,1,2,\cdots).\label{39a}
\end{align}
From (\ref{6.2}), it holds (\ref{39a}) with $j=0$.
We assume that (\ref{39a}) holds for one natural number $j$
and $(r,t) \in \Sigma_0$. Noticing that $Q_{0}(r,t) \subset \Sigma_0$ for $(r,t) \in \Sigma_0$
 and putting (\ref{39a})
into (\ref{6.1}), we get
\begin{align}
	\tilde{u}(r, t) 
	&\geq 
 	\iint_{Q_0(r, t)} \frac{\lambda}{t+r} (1+s)^{-(p-1)} 
 	\left| \frac{D_{j}(s-\lambda-\rho)^{a_j}}{(s+\lambda)(s-\lambda)^{b_j}}  \right|^p 
 	d\lambda ds \quad \mbox{in} \ \Sigma_0.\label{6.3}
\end{align}
Then using the variable of integration $\alpha, \beta$ from (\ref{dai2-10}), we obtain
\begin{align}\label{6.4}
	\widetilde{u}(r, t) 
	&\geq 
	\frac{1}{2(t+r)} \int_{\rho}^{t-r} d\beta \int_{2(t-r) + \beta}^{3(t-r)} 
	\frac{\alpha - \beta}{2} \left(\frac{\alpha + \beta + 2}{2}\right)^{1-p} 
	\left|\frac{D_j(\beta-\rho)^{a_j}}{\alpha \beta^{b_j}} \right|^p d\alpha \nonumber \\
	&\geq \frac{D_j^p}{4(t+r)}((6(t-r))^{-(p-1)}
	\int_{\rho}^{t-r} d  \beta \beta^{-pb_j}(\beta-\rho)^{pa_j}
	\int_{2(t-r)+\beta}^{3(t-r)}  \frac{\alpha-\beta}{\alpha^p} d\alpha
	\nonumber\\
	&\geq \frac{D_j^p}{4(t+r)} (6(t-r))^{-(p-1)}(t-r)^{-pb_j}
	\int_{\rho}^{t-r} d\beta
	(\beta - \rho)^{pa_j} 
	 \int_{2(t-r) + \beta}^{3(t-r)} \frac{2(t-r)}{3^p(t-r)^p}
	 d\alpha \nonumber \\
	&= \frac{ 18^{-p}  D_j^p}{2(t+r)(t-r)^{pb_j+2(p-1)} } 
	\int_{\rho}^{t-r} (\beta - \rho)^{pa_j}(t-r-\beta) d\beta.
\end{align}
Using integration by parts, it follows that
\begin{align}\label{6.5}
	\int_{\rho}^{t-r} (\beta - \rho)^{pa_j}(t-r-\beta) d\beta
	&= \frac{1}{pa_j+1}\int_{\rho}^{t-r} (\beta-\rho)^{pa_j+1} d\beta \nonumber \\ 
	&\geq \frac{1}{a_{j+1}^2}(t-r-\rho)^{pa_j+2}.
\end{align}
Therefore, from (\ref{6.4}) and (\ref{6.5}), (\ref{39a}) holds for all natural number.

Solving (\ref{40b}) and (\ref{40c}) yields 
\begin{align}
	a_j = \frac{p^{j}(p+1)-2}{p-1}, \quad b_j = 2(p^{j+1}-1).
	\quad (j=0,1,2,\cdots).\label{dai5-2}
\end{align}
Hence, we get from (\ref{40c}) and (\ref{dai5-2})
\begin{align*}
	D_{j+1} \geq \frac{F D_j^p}{p^{2(j+1)}},
\end{align*}
where $F=\frac{18^{-p}}{2} (\frac{p-1}{p+1})^2$.
Hence we have
\begin{align}\label{dai5-5}
	\log{ D_{j} } \geq
	p^j\Bigl[ \log{D_0} + \sum_{k=1}^{j} 
	\frac{p^{k-1}\log{F} -2k \log{p} }{p^j} \Bigr].
\end{align}
By using the d'Alembert's criterion, we see that the sum part in (\ref{dai5-5}) converges as $j \to \infty$.
Hence, from (\ref{40b}), there exists a constant $q$ such that it holds
\begin{align}
	D_{j}  \geq \exp\{p^j\log{(M e^q \epsilon^p )}\}.\label{dai5-4}
\end{align}
Therefore, we have from (\ref{39a}), (\ref{dai5-2}) and (\ref{dai5-4})
\begin{align}\label{41}
	\widetilde{u}(r, t)\geq  
	\frac{ \exp{[p^j J(r, t)]} (t-r)^{2}}{(t+r)(t-r-\rho)^{\frac{2}{p-1}}} 
	\quad \mbox{in} \ \Sigma_0.
\end{align}
Here,  
\begin{align}\label{42}
	J(r, t) = \log \left\{ \epsilon^p M e^{q} 
	\frac{(t-r-\rho)^{\frac{p+1}{p-1}}}{(t-r)^{2p}} \right\}.
\end{align}

We take $\epsilon_0>0$ so small that
\begin{align}
	B \epsilon_{0}^{-\frac{2p(p-1)}{\gamma(p,5)}} \geq 8 \rho,\nonumber
\end{align}
where
\begin{align}
	B=\left(2^{\frac{2(p^2-2p-1)}{p-1}}Me^q\right)^{-\frac{2(p-1)}{\gamma(p,5)}}.\nonumber
\end{align}
Next, for a fixed $\epsilon \in (0,\epsilon_0]$, we suppose that
\begin{align}
	\tau > B \epsilon^{-\frac{2p(p-1)}{\gamma(p,5)}} \geq 8\rho.\label{dai5-6}
\end{align}
Let $(r,t)=(\tau/2,\tau)$. Then $(r,t) \in \Sigma_0$
and $t-r-2\rho \geq (t-r)/2$. Hence we get from (\ref{42}) and (\ref{dai5-6})
\begin{align}
	J(\tau/2,\tau) &\geq \log \left(
	\epsilon^p (B^{-1}\tau)^{\frac{\gamma(p,5)}{2(p-1)}}
	\right)
	>0.\label{dai5-7}
\end{align}
Therefore, from (\ref{41}) and (\ref{dai5-7}), we get $\tilde{u}(\tau/2,\tau) \to \infty$ 
$(j \to \infty)$. Hence, $T(\epsilon) \leq B \epsilon^{-\frac{2p(p-1)}{\gamma(p,5)}}$
for $0< \epsilon \leq \epsilon_0$.

\noindent (ii) \ Estimation in the case of $p=p_{S}(5)$.

We define the sequences $\{ d_j \}$ and $\{ E_j \}$ by
\begin{align}
	d_0=0,\quad E_{0}=M \epsilon^p\label{6.9}
\end{align}
and
\begin{align}
	&d_{j+1}=p d_{j}+1,\quad E_{j+1}= \frac{18^{-p}   E_j^p  }{2^{j+3} d_{j+1} } 
	\quad (j=0,1,2,\cdots).\label{6.8}
\end{align}

First, by using the induction argument, we will show
\begin{align}\label{6.6}
	\widetilde{u}(r, t) \geq 
	\frac{E_j}{(t+r)(t-r)^{2p-3}} \Bigl(\log{\frac{t-r}{l_j\rho}}\Bigr)^{d_j} \quad
	\mbox{in} \ \Sigma_j
	\quad (j=0,1,2,\cdots).
\end{align}
From (\ref{6.2}), it holds (\ref{6.6}) with $j=0$.
We assume that (\ref{6.6}) holds for one natural number $j$ and $(r,t) \in \Sigma_{j+1}$.
From $p=p_{S}(5)$, we have $p(2p-3)=1$.
Noticing (\ref{dai5-18}) and substituting (\ref{6.6}) into (\ref{6.1}), we obtain
\begin{align}
	\widetilde{u}(r, t) &\geq 
	\frac{1}{2(t+r)} \int_{l_j\rho}^{t-r} d\beta
	\int_{2(t-r) + \beta}^{3(t-r)} \frac{\alpha-\beta}{2}
	\Bigl( \frac{2+\alpha+\beta}{2} \Bigr)^{1-p} 
	\Biggl|\frac{E_j}{\alpha \beta^{2p-3}}
	\Bigl(\log{\frac{\beta}{l_j\rho}} \Bigr)^{d_j} \Biggr|^p d\alpha\nonumber\\
	&\geq \frac{E_j^p}{2(t+r)} \int_{l_j \rho}^{t-r} d\beta
	\int_{2(t-r) + \beta}^{3(t-r)}
	\frac{6^{1-p}(t-r)^{1-p} }{3^{p}(t-r)^{p-1}
	\beta^{p(2p-3)} }\Bigl(\log{\frac{\beta}{l_j\rho}} \Bigr)^{pd_j}  d\alpha\nonumber\\
	&\geq \frac{ 18^{-p} E_j^p }{2(t+r)(t-r)^{2(p-1)} } \int_{l_j\rho}^{t-r} 
	\frac{t-r-\beta}{\beta} \Bigl( \log{\frac{\beta}{l_j\rho}} \Bigr)^{pd_j}  d\beta.
	\label{dai5-9}
\end{align}
Since $\rho \leq \frac{t-r}{l_{j+1} }$,
we get from (\ref{dai5-9}) and (\ref{6.8})
\begin{align*}
	\widetilde{u}(r, t) &\geq 
	\frac{ 18^{-p}  E_j^p }{2(t+r)(t-r)^{2(p-1)} }
	\int_{l_{j}\rho}^{t-r} (t-r-\beta) \Biggl[ \frac{1}{p d_j+1}
	\Bigl( \log{\frac{\beta}{l_{j}\rho}} \Bigr)^{pd_j+1}\Biggr]' 
	d\beta \\
	&= \frac{ 18^{-p}  E_j^p }{2(t+r)(t-r)^{2(p-1)}}
	\times \frac{1}{p d_{j}+1}\int_{l_j \rho}^{t-r} 
	\left(\log\frac{\beta}{l_j \rho}\right)^{pd_j+1}d\beta\\
	&\geq \frac{ 18^{-p}  E_j^p}{2 d_{j+1}(t+r)(t-r)^{2(p-1)} }
	\int_{\frac{l_{j}(t-r)}{l_{j+1}}}^{t-r} 
	\left(\log\frac{\beta}{l_j \rho}\right)^{d_{j+1}}d\beta\\
	&\geq \frac{ 18^{-p}  \Bigl( 1-\frac{l_j}{l_{j+1}} \Bigr) E_j^p }{2d_{j+1}(t+r)(t-r)^{2p-3} }
	\Bigl( \log{\frac{t-r}{l_{j+1}\rho}} \Bigr)^{d_{j+1}}.
\end{align*}
From $1-\frac{l_j}{l_{j+1}}  = \frac{1}{2^{j+1}l_{j+1} } \geq \frac{1}{2^{j+2}}$, we get
\begin{align*}
	\widetilde{u}(r, t) 
	&\geq \frac{ 18^{-p}  E_j^p}{2^{j+3} d_{j+1} (t+r)(t-r)^{2p-3} }
	\Bigl( \log{\frac{t-r}{l_{j+1}\rho}} \Bigr)^{d_{j+1}} \nonumber \\
	&= \frac{E_{j+1}}{(t+r)(t-r)^{2p-3} }	
	\Bigl( \log{\frac{t-r}{l_{j+1}\rho}} \Bigr)^{d_{j+1}} \quad \mbox{in} \ \Sigma_{j+1}.
\end{align*}
Therefore, (\ref{6.6}) holds for all natural number.

Solving (\ref{6.9}) and (\ref{6.8}) yields
\begin{align}
	d_j = \frac{p^j-1}{p-1}
	\quad (j=0,1,2,\cdots).\label{dai5-8}
\end{align}
Hence, we get
\begin{align*}
	E_{j+1} \geq \frac{G E_j^p}{(2p)^j},
\end{align*}
where $G=\frac{18^{-p} (p-1)}{2^3p}$.
Therefore, it follows that
\begin{align}\label{3.20}
	\log{ E_{j} } \geq
	p^j\Bigl[ \log{E_0} + \sum_{k=1}^{j} 
	\frac{p^{k-1}\log{G} - (k-1)\log{(2p)} }{p^j} \Bigr].
\end{align}
The sum part in (\ref{3.20}) converges as $j \to \infty$ by the d'Alembert's criterion.
Hence, there exists a constant $q$ such that it holds from (\ref{6.9})
\begin{align*}
	E_{j}  \geq \exp\{p^j\log{(M e^q \epsilon^p )}\}.
\end{align*}
Since $\Sigma_{\infty} \subset \Sigma_j$ and $l_{j} \leq 2$, we obtain from 
(\ref{6.6}) and (\ref{dai5-8})
\begin{align}\label{3.21}
	\widetilde{u}(r, t)
	&\geq \frac{\exp\{p^j\log{(M e^q  \epsilon^p) } \}}{(t+r)(t-r)^{2p-3}}
	\Bigl(\log{\frac{t-r}{2\rho}} \Bigr)^{\frac{p^j-1}{p-1}}\nonumber\\
	&\geq \frac{\exp\{p^jJ(r, t) \}}{(t+r)(t-r)^{2p-3}}
	\Bigl(\log{\frac{t-r}{2\rho}} \Bigr)^{-\frac{1}{p-1}} \quad \mbox{in} \ \Sigma_{\infty},
\end{align}
where
\begin{align}\label{3.24}
	J(r, t)
	= \log{\Biggl( \epsilon^p
	\Bigl( B^{-1 } \log{\frac{t-r}{2 \rho}\Bigr)^{\frac{1}{p-1}} } \Biggr) },
	\quad B=(Me^q)^{-(p-1)} .
\end{align}

We take $\epsilon_0>0$ so small that 
\begin{align}
	B \epsilon_0^{-p(p-1)} \geq \log(4 \rho).\label{dai5-10}
\end{align}
For a fixed $\epsilon \in (0,\epsilon_0]$, we suppose that $\tau$ satisfies
\begin{align}
	\tau > \exp(2B \epsilon^{-p(p-1)}) \ (> 4\rho).\label{dai5-11}
\end{align}
From (\ref{dai5-11}) and (\ref{dai5-10}), it follows that
\begin{align}
	\tau > 4 \rho \exp \left(B \epsilon^{-p(p-1)}\right).\label{dai5-12}
\end{align}
We get (\ref{3.24}) and (\ref{dai5-12})
\begin{align}
	J(\tau/2,\tau)=\log \left( \epsilon^{p} 
	\left( B^{-1} \log \frac{\tau}{4\rho} \right)^{\frac{1}{p-1}}\right) >0.\label{dai5-13}
\end{align}
Since $(\tau/2,\tau) \in \Sigma_{\infty}$,
from (\ref{3.21}) and (\ref{dai5-13}), we get $u(\tau/2,\tau) \to \infty$ \ 
$(j \to \infty)$. Hence, $T(\epsilon) \leq \exp(2B \epsilon^{-p(p-1)})$ for 
$0< \epsilon \leq \epsilon_0$.
This completes the proof.
\end{proof}

\begin{acknowledement}
The authors would like to express his sincere gratitude to Professor Hiroyuki Takamura
for his valuable advices.
\end{acknowledement}

\begin{flushright}
\bigskip
\address{M. KATO\\
Muroran Institute of Technology\\
27-1 Mizumoto-cho, Muroran 050-8585, Japan
}
{mkato@mmm.muroran-it.ac.jp}
\address{M. SAKURABA\\
Sapporo Hokuto High School\\
1-10 Kita 15, Higashi 2, Higashi-ku, Sapporo 065-0015, Japan
}
{miku.sakuraba@gmail.com}

\end{flushright}

\end{document}